\documentclass[12pt]{article}
\usepackage{amsmath}
\usepackage{amsfonts}
\usepackage{amssymb}
\usepackage{amsthm}
\usepackage{stackrel}

\usepackage{color}

\newtheorem{thm}{Theorem}[section]
\newtheorem{co}{Corollary}[section]
\newtheorem{re}{Remark}[section]

\def\binom#1#2{{#1}\choose{#2}}
\newcommand{\R}{{\rm I}\kern-0.18em{\rm R}}
\newcommand{\1}{{\rm 1}\kern-0.25em{\rm I}}
\newcommand{\E}{{\rm I}\kern-0.18em{\rm E}}
\newcommand{\p}{{\rm I}\kern-0.18em{\rm P}}
\makeatletter

\title{Distributional Tail Estimation Through Its Characteristic Function}
\author{Lev B. Klebanov\footnote{Department of Probability and Mathematical Statistics, Charles University, Prague, Czech Republic. e-mail: levbkl@gmail.com} and Andrea Karlov\'{a}\footnote{e-mail:   andrea.karlova@gmail.com}}
\date{}
\begin{document}
\maketitle	

\begin{abstract}
There is given a method for estimation of a probability distribution tail in terms of characteristic function.

\end{abstract}

\section{Introduction}
There are large classes of probability distributions defined by their characteristic functions. It is enough to mention: 
\begin{itemize}
\item stable distributions (see \cite{Zolot}); 
\item $\nu$-stable distributions (see \cite{KKozR}); 
\item discrete stable distributions (see \cite{KSl}); 
\item tempered stable distributions (see \cite{Tempered}). 
\end{itemize}
 Below we give a method obtaining estimators for the tails of arbitrary distributions basing on their characteristic functions. 

\section{Main results}
\setcounter{equation}{0}

\begin{thm}\label{th1} Let 
\begin{equation}\label{eq1}
P(\theta)= \sum_{j=0}^k a_k \cos(j\theta) + \sum_{j=1}^k b_k \sin(j\theta)
\end{equation}
be a non-negative trigonometric polynomial of degree $k$. Suppose that $F(x)$ is a cumulative distribution function, and $f(t)$ is its characteristic function. Then for any $s>0$
\begin{align}\label{eq2}
\begin{aligned}
1-F(2\pi /s)+ F(-2\pi /s) \le \qquad \qquad \\ \le \frac{2}{s a_0}\Bigl(\sum_{j=0}^k a_j \int_{0}^{s}\tt{Re}\, f(ju)du + \sum_{j=1}^k b_j \int_{0}^{s}\tt{Im}\, f(ju)du \Bigr).
\end{aligned}
\end{align}
\end{thm}
\begin{proof}
Denote $F_1(x)=F(x)-F(-x)$. Let us consider
\begin{align}\label{eq3}
\begin{aligned}
\int_{0}^{s}\Bigl(\int_{-\infty}^{\infty} P(xu)dF(x) \Bigr)du=\int_{0}^{s}\Bigl(\int_{0}^{\infty} P(xu)dF_1(x) \Bigr)du \\
= \int_0^{\infty}\Bigl(\int_0^sP(xu)du \Bigr)dF_1(x).\qquad \qquad
\end{aligned}
\end{align}
We have
\begin{align}\label{eq4}
\begin{aligned}
\int_{0}^{s}\Bigl(\int_{-\infty}^{\infty} P(xu)dF(x) \Bigr)du=\qquad \qquad \\ =\sum_{j=0}^k a_j \int_{0	}^{s}\tt{Re}\, f(ju)du + \sum_{j=1}^k b_j \int_{0}^{s}\tt{Im}\, f(ju)du,
\end{aligned}
\end{align}
which is identical with the expression in brackets at right hand side of (\ref{eq2}).

On the other hand, 
\begin{align}\label{eq5}
 \int_0^{\infty}\Bigl(\int_0^sP(xu)du \Bigr)dF_1(x) =
 \int_0^{\infty}\Bigl( \frac{1}{x}\int_{0}^{sx}P(v)dv \Bigr)dF_1(x)
\end{align}
It is clear that the polynomial $P$ has $2\pi$ as its period (not necessary, minimal one). Let us define numbers $A_m =2\pi m/s$, $m=1,2, \ldots$. It is also clear that 
\[ \int_0^{2\pi}P(\theta)d \theta = 2 \pi a_0. \]
From relations (\ref{eq3}) and (\ref{eq5}) it follows that

\begin{align*}
\begin{aligned}
\int_{0}^{s}\Bigl(\int_{-\infty}^{\infty} P(xu)dF_1(x) \Bigr)du \ge \sum_{m=1}^{\infty}\int_{A_m}^{A_{m+1}}\frac{1}{x}\Bigl(\int_{0}^{s}P(v)dv\Bigr)dF_1(x) \ge\\
\ge \sum_{m=1}^{\infty}\frac{1}{A_{m+1}}\int_{0}^{A_m s}P(v) dv \Bigl( F_1(A_{m+1}) - F_1(A_m)\Bigr)\ge \\
\end{aligned}
\end{align*}
\begin{align}\label{eq6}
\begin{aligned}
\ge \sum_{m=1}^{\infty}\frac{1}{A_{m+1}}\Bigl( \sum_{j=0}^{m-1}\int_{A_j s}^{A_{j+1}s}P(v) dv\Bigr)\Bigl(F_1(A_{m+1}-F_1(A_m))\Bigr) \ge \\
\ge \sum_{m=1}^{\infty}\frac{sm}{2\pi (m+1)}2\pi a_0 \Bigl(F_1(A_{m+1}-F_1(A_m))\Bigr) \ge \\
\ge \frac{s a_0}{2}\Bigl(1-F_1(A_1)\Bigr)= \frac{s a_0}{2}\Bigl(1-F(2 \pi /s)+F(-2 \pi /s)\Bigr).
\end{aligned}
\end{align}
The result follows now from (\ref{eq4}) and (\ref{eq6}).
\end{proof}
Choosing different trigonometric polynomial $P(\theta)$ we obtain variety of tail estimators. Namely, for the case of
\[ P(\theta)= \sin^{2k}\theta =\frac{1}{2^{2k}} {\binom{2k}{k}}+\frac{1}{2^{2k-1}}\sum_{j=0}^{k-1}(-1)^{k-j}{\binom{2k}{j}}\cos(2(k-j)\theta)\]
we obtain the following 
\begin{co}\label{co1} (see \cite{Sap})
For a distribution function $F$ holds the following estimator
\begin{align}\label{eq7}
F(-2\pi/s)+1-F(2\pi/s) \le \frac{(-1)^k 2^((2k)!!)}{4^k (2k-1)!!}\int_{0}^{s}\Delta_u^{(2k)}({\tt Re}\,f,0)du,
\end{align}  
where
\[ \Delta_u^{(2k)}(g,t) = \sum_{j=0}^{2k}(-1)^j{\binom{2k}{j}}g(t-(k-j)u)\]
and $s>0$.
\end{co}
\begin{re}\label{re1}
Let us note that if characteristic function $f(t)$ has derivative of order $2k$ at the origin, then
\begin{equation}\label{eq8}
|\Delta_u^{(2k)}({\tt Re}\,f,0)| \le u^{2k} |f^{(2k)}(0)|.
\end{equation}
\end{re}
It is easy to see that Corollary \ref{co1} gives 
correct order of the tails for symmetric stable and geometric stable distributions as $s \to 0$.

Theorem \ref{th1} and its corollary work nicely for the case of distributions with tails of power order. However, many probability laws have exponential tails. This is the case for, say, tempered stable distributions. Therefore, it is interesting to obtain exponential boundaries for the tails of analytical characteristic functions.

\begin{thm}\label{th2}
Suppose that $F(x)$ is probability law, whose characteristic function $f(t)$ is analytic for $|t|<R$ ($0<R \le \infty$). Then for any $A>0$ and arbitrary $s\in (0,R)$ the following inequality holds:
\begin{align}\label{eq9}
1-F(A)+F(-A) \le \frac{\int_{0}^{s}\Bigl(\frac{f(i u)+f(-iu)}{2} -1 \Bigr)du}{s\,\bigl(\sinh(As)/(As)-1\bigr)}.
\end{align}
\end{thm}
\begin{proof}
From Raikov Theorem (see, for example, \cite{Lin}) it follows that characteristic function $f(t)$ is analytic in the strip $|\tt{ Im}t |<R$, and the distribution $F(x)$ has exponential moments of all orders less than $R$. 
Consider the following expression:
\begin{align}\label{eq10}
\int_0^s \Bigl(\int_{0}^{\infty}(\cosh(xu)-1)dF_1(x)\Bigr)du = \int_0^s \Bigl(\frac{f(iu)+f(-iu)}{2}-1\Bigr)du,
\end{align}
where, as before, $F_1(x)=F(x)-F(-x)$. Left hand side of the equation (\ref{eq10}) may be transformed in the following way:
\begin{align}\label{eq11}
\begin{aligned}
\int_0^s \Bigl(\int_{0}^{\infty}(\cosh(xu)-1)dF_1(x)\Bigr)du =\int_{0}^{\infty}\Bigl(\int_0^s(\cosh(xu)-1)du \Bigr)dF_1(x)=\\ = s \int_{0}^{\infty}\Bigl(\frac{\sinh(sx)}{sx}-1\Bigr)dF_1(x) \ge s \int_{A}^{\infty}\Bigl(\frac{\sinh(sx)}{sx}-1\Bigr)dF_1(x) \ge \\ \ge s \, \Bigl(\frac{\sinh(sA)}{sA}-1\Bigr)(1-F_1(A)).
\end{aligned}
\end{align}
Inequalities (\ref{eq11}) and (\ref{eq10}) leads to (\ref{eq9}).
\end{proof}

It is interesting to note the following consequence of Theorem \ref{th2}.

\begin{co}\label{co2}
Suppose that characteristic function $f(t)$ is an entire of a finite exponential type. Then corresponding distribution function $F(x)$ is concentrated on a compact subset of real line.
\end{co}
\begin{proof}
Let us remind that entire function $f(t)$ has finite exponential type not greater that $\rho$ if 
\begin{equation}\label{eq12}
\lim \sup_{r \to \infty}\frac{\log M(r)}{r} \le \rho,
\end{equation}
where
\[ M(r) =\max_{|z|\le r}|f(z)|.\]
For any $\rho_1 >\rho$ there is a positive constant $C$ such that
\[ M(r)< C \exp(\rho_1 r)\]
and, for sufficiently large $s>0$ integral in the right hand side of (\ref{eq9}) is smaller than $C_2 s \exp(\rho_1 s)$, where $C_2>0$ is a positive constant. For $A>\rho_1$ right hand side of (\ref{eq9}) tends to zero as $s$ tends to infinity. This shows that the left hand side of (\ref{eq9}) is zero for any $A>\rho$, and the support of distribution function $F(x)$ is concentrated in interval $[-\rho,\rho]$.
\end{proof}
Of course, if $F(x)$ is concentrated on compact subset of real line then its characteristic function is entire of finite exponential type.

Let us note that Corollary \ref{co2} is a consequence of Paley-Wiener Theorem (see \cite{PW}).

\section{One-sided estimators for the tail}
\setcounter{equation}{0}

Here we give one-sided estimators for the tail of distributions having characteristic function analytic in a region containing interval of the form $ t \in (0,a\,i)$ or $t \in (-b\,i,0)$ ($a,b >0$) on imaginary axis. In this situation, characteristic function is analytic in the strip $0<{\tt Im}\,t < a$ or $-b< {\tt Im}\,t<0$ correspondingly, and may be continued as corresponding integral (see \cite{LO}).

\begin{thm}\label{th3}
Suppose that $F(x)$ is probability law, whose characteristic function $f(t)$ is analytic in the strip $-b< {\tt Im}\,t<0$ (correspondingly, $0< {\tt Im}\,t<b$) for a positive $b$. Then for any $A>0$ and arbitrary $s \in (0,b)$ the following inequality holds:
\begin{equation}\label{eq13}
1-F(A) \le \frac{A}{\exp(sA)-sA-1}\int_{0}^{s}(f(-iu)+F(+0)-1)du,
\end{equation}
(correspondingly
\begin{equation}\label{eq14}
F(-A) \le \frac{A}{\exp(sA)-sA-1}\int_{0}^{s}(f(iu)-F(-0))du\; ).
\end{equation}
\end{thm}
\begin{proof}
It is enough to repeat the proof of Theorem \ref{th2} using the function $e^{ux}-1$ (correspondingly, $e^{-ux}-1$) instead of $cosh(xu)-1$.
\end{proof}
Suppose that $b=\infty$. From (\ref{eq13}) (correspondingly, from (\ref{eq14})) it follows that if $f(-iu) \le \exp (au)$  (correspondingly,  $f(iu) \le \exp (au)$) for sufficiency large $u$, then corresponding tail is zero for large values of $A$. This is known result (see \cite{Po}).

Let us note that the estimators (\ref{eq13}) and (\ref{eq14}) do not work for small values of $s$, that is as $s \to 0$. This is an essential difference with the cases of (\ref{eq2}) and (\ref{eq9}). 
\section{Acknowledgment}
The work was partially supported by Grant GA\u{C}R 16-03708S.

\end{document}